\documentclass[reqno,a4paper]{amsart}

    \usepackage[utf8]{inputenc}
    \usepackage{fullpage}
    \usepackage{setspace}
    \usepackage{microtype}
    \usepackage[dvipsnames]{xcolor}
    \usepackage[colorlinks = true, linkbordercolor = {white}, urlcolor={purple}, linkcolor={purple}, citecolor = {purple}]{hyperref}
    \usepackage{cleveref}
    \usepackage{biblatex}
    \usepackage{mathrsfs}
    \usepackage{amsmath}
    \usepackage{amsthm}
    \usepackage{amssymb}
    \usepackage{csquotes}
    \usepackage{float}
    \usepackage{mathtools}
    \usepackage{accents}
    \usepackage{tikz}
        \usetikzlibrary{patterns}
        \usetikzlibrary{graphs,graphs.standard,quotes}
        \usetikzlibrary{arrows}
    \usepackage{soul}
    \usepackage[foot]{amsaddr}
    \usepackage{marginnote}
    \usepackage{enumitem}

    \newcommand{\R}{\mathbb{R}}

    \newcommand{\N}{\mathbb{N}}
    \newcommand\open[1]{\accentset{\circ}{#1}}

    \theoremstyle{definition}
        \newtheorem*{nonotheorem}{Theorem}
        \newtheorem{thrm}{Theorem}[section]
        \newtheorem{lem}[thrm]{Lemma}
        \newtheorem{prop}[thrm]{Proposition}
        \newtheorem*{nonoprop}{Proposition}
        \newtheorem{cor}[thrm]{Corollary}
        \newtheorem*{nonocor}{Corollary}
        
        \newtheorem{defin}[thrm]{Definition}

    \theoremstyle{remark}
        \newtheorem{rem}[thrm]{Remark}
    \theoremstyle{plain}

\title{Affine thickness: Patterns and a Gap Lemma}

\date{\today}

\allowdisplaybreaks

\begin{document}

\author[Richard\,Alasdair\,Howat]{Richard Alasdair Howat}

\address[R.\,A.\,Howat]{School of Mathematics, University of Birmingham, UK}

\begin{abstract}
    A new notion of thickness for subsets of $B[0,1]\subset \mathbb{R}^n$ called affine thickness is defined; this notion of thickness is a generalisation of Falconer-Yavicoli thickness and is adapted to be used in the study of certain sets with affine cut outs.
    Thick sets are proven to be winning for the matrix potential game introduced in \autocite{howat2025matrixpotentialgamestructures} and as an application we can prove that for a thick set, there exists $M\in\mathbb{N}$ depending on the thickness of the set, such that the set contains a homothetic copy of every finite set with at most $M$ elements.
    Additionally, the author provides a counter-example to the gap lemma in $\mathbb{R}^n$ ($n\geq 2$) for Falconer-Yavicoli thickness, stated in \autocite[Theorem 10]{falconer2022intersections-ThicknessBasedGame} proving this result does not hold in the generality stated.
    We go on to provide a gap lemma for affine thickness in $\mathbb{R}^n$ (for $n\geq 2$) under additional conditions to the classical Newhouse gap lemma.
\end{abstract}

\maketitle

\section{Introduction}
Given a compact set $C\subseteq B[0,1]$, we let $(G_k)_{k\in J}$ denote the at most countably many open, bounded, path-connected components of $\mathbb{R}^n\setminus C$ and let $E$ be the unbounded connected component (in $\R$ the set $E$ is the union of two unbounded components).
We call these our gaps.
Using this structure, in \autocite[]{falconer2022intersections-ThicknessBasedGame}, Falconer and Yavicoli defined their notion of thickness in $\R^n$ for the standard Euclidean distance, which we denote $d$.
Extending the definition of Newhouse thickness for subsets of the real line, as presented in \autocite[]{NewhouseThickness}.
It is defined as follows:
\begin{defin}\label{Def: YAV-FALC-NEWH thickness}
    For a compact set $C\subseteq \R^n$, assuming the gaps are ordered in non-increasing diameter, the thickness of $C$ is defined to be  
    \begin{align}\label{eq: Falconer-Yavicoli thickness}
        \tau(C)=\inf_{k\in J}\left\{ \frac{d\left(G_k,E\cup\bigcup_{j=1}^{k-1} G_j\right)}{\operatorname{diam}(G_k)} \right\}.
    \end{align}
    When $J=\emptyset$, that is when there are no bounded gaps, then
    \begin{equation*}\label{eq: Thickness equation when no bounded}
        \tau(C)=\begin{cases}
            \infty \text{ \,\,\, if \,\,\, } \operatorname{int}(C)\neq\emptyset \\
            0 \text{ \,\,\,\,\,\, if \,\,\, } \operatorname{int}(C)=\emptyset
        \end{cases}.
    \end{equation*}
\end{defin}
We note that this notion of thickness is often zero when the gaps are defined using an affine, non-homothetic map, for example in self-affine Bedford-McMullen carpets.
To overcome this, we look to mimic an approach presented by the author, Mitchell and Samuel in \autocite{howat2025matrixpotentialgamestructures}; that is view the set through the lens of an affine map.

Through this technique we can define a new notion of thickness, which we refer to as thickness with respect to $A$, where $A$ is a diagonal affine matrix with positive entries or when the matrix is clear, affine thickness.
We note that our notion is defined for subsets of $B[0,1]$ rather than compact subsets of $\R^n$, this is not restrictive as one can scale the bounded set into $B[0,1]$ using some power of the affine matrix $A$.
This will not affect the thickness as all defined values will be affected equally.
Moreover, we define this notion for both the Euclidean metric and the square metric on $\R^n$. 
\begin{defin}\label{def: size wrt matrix}
    Consider the square or Euclidean metric on $\R^n$ and let $A$ be a diagonal matrix with entries $\beta_{11},\dots,\beta_{nn}\in (0,1)$.
    For a set $F\subseteq B[0,1]$ we define the size with respect to the matrix $A$ by
    \begin{align}\label{eq: size wrt matrix}
        S_A(F) =\inf \left\{t>0: \exists z\in\R^n \text{ s.t } F\subseteq A^{1/t}(B[0,1])+z\right\}.
    \end{align}
\end{defin}

We now define our notion of thickness with respect to a diagonal matrix $A$.

\begin{defin}\label{def: affine thickness and gap distance def}
    Consider the square or Euclidean metric on $\R^n$, let $A$ be a diagonal matrix with entries $\beta_{11},\dots,\beta_{nn}\in (0,1)$, and let $C\subseteq B[0,1]$ be a compact subset with gaps $(G_k)_{k\in J}$ and $E$.
    Assuming $(G_k)_{k\in J}$ is ordered in non-increasing size with respect to $A$, for $m\in\N$ we define the gap distance with respect to the matrix $A$ by
    \begin{align}\label{eq: gap distance equation}
        GD_A(m,C)=\inf\left\{t>0:\substack{\exists z\in\R^n \text{ s.t } G_m\cap A^{1/t}(B[0,1])+z \neq \emptyset \text{ and } \\ \left(A^{1/t}(B[0,1])+z\right)\cap \left( \bigcup_{i < m} G_i\cup E\right)\neq\emptyset}\right\},
    \end{align}
        and thickness with respect to $A$ by
    \begin{align}\label{eq: Affine Thickness equation}
        \tau_A(C)=\inf_{k\in J}\left\{ S_A(G_k)^{-1}-GD_A(k,C)^{-1} \right\}.
    \end{align}
    If $GD_A(k,C)=0$ for some $k\in\N$, then we set $\tau_A(C)=-\infty$ and if $J=\emptyset$, then
    \begin{equation}\label{eq: Affine Thickness equation when no bounded}
        \tau_A(C)=\begin{cases}
            \infty \text{ \,\,\,\,\,\,\,\,\,\,\, if \,\,\, } \operatorname{int}(C)\neq\emptyset \\
            -\infty \text{ \,\,\,\,\,\, if \,\,\, } \operatorname{int}(C)=\emptyset
        \end{cases}.
    \end{equation}
\end{defin}
\begin{rem}\label{rem: our thickness is similar to Yavicoli-Falconer thickness}
    For a compact set $C$ and matrix $A=\beta I$, we have the thickness $\tau(C)$ defined in \Cref{Def: YAV-FALC-NEWH thickness} is such that if $\tau(C)>0$, the following holds; $\log_{\beta}(\tau(C))=-\tau_A(C)$.
    This can be seen by rewriting the metric distance between two sets and gap diameter in terms of gap distance and size with respect to the matrix $A$.
    This shows the our notion of thickness is closely related when the matrix is homothetic.
\end{rem}
\begin{rem}
    We note that gaps are ordered in non-increasing size and the order of the gaps of the same size does not effect the value of thickness, hence thickness is well defined.
    This is because when considering the finitely many gaps of the same size with respect to $A$, the smallest gap distance will appear in the minimum at some stage hence the thickness stays consistent.
\end{rem}

Newhouse's gap lemma states that if you are given two compact sets $C_1,C_2\subseteq \R$ such that neither set lies in a gap of the other and $\tau(C_1)\tau(C_2)>1$, then $C_1\cap C_2\neq\emptyset$.
Indeed, one would hope that this extends directly to thickness defined in $\R^n$, for $n\geq 2$, for both affine and Falconer-Yavicoli thickness.
However, as stated in the following proposition, we can find examples of pairs of sets with large affine thickness whose intersection is empty.

\begin{nonoprop}\label{nonoprop: counter examp}
   Consider the square or Euclidean metric on $\R^n$ and $n\geq 2$.
   Let $A$ be a diagonal matrix with entries $\beta_{11},\dots,\beta_{nn}\in (0,1)$.
   There exists compact sets $C_1,C_2\subseteq B[0,1]$ such that neither of them is contained in a gap of the other, $\tau_A(C_1)+\tau_A(C_2)>0$
   and $C_1\cap C_2=\emptyset$. 
\end{nonoprop}
This directly leads to the following corollary, providing a counterexample to the gap lemma for $\R^n$ ($n\geq 2$) stated by Falconer and Yavicoli in \autocite[Theorem 10]{falconer2022intersections-ThicknessBasedGame}.
Showing that in fact the gap lemma does not directly extend to higher dimensions with Falconer-Yavicoli thickness.
\begin{nonocor}\label{nonocor:CounterexampletoTHRM10}
    There exists sets $C_1,C_2\subseteq \R^n$ ($n\geq 2$) such that neither set lies in a gap of the other, $\tau(C_1)\tau(C_2)>1$ and $C_1\cap C_2=\emptyset$.
\end{nonocor}
Since the gap lemma does not directly extend to higher dimensions we must explore conditions in which an affine gap lemma is true in $\R^n$.
This leads to \Cref{def: strongrefinable}, in which we define strongly refinable pairs of sets. 
Using this definition we can obtain an affine gap lemma, stated below. 
\begin{nonotheorem}[Affine Gap Lemma]\label{nonothrm: Affine Gap Lemma}
    Consider the square or Euclidean metric on $\R^n$, let $A$ be a diagonal matrix with entries $\beta_{11},\dots,\beta_{nn}\in (0,1)$ and let $C_1,C_2\subseteq B[0,1]$ be a strongly refinable pair, non-empty and compact. If $\tau_A(C_1)+\tau_A(C_2)>0$, then $C_1\cap C_2\neq\emptyset$.
\end{nonotheorem}
The connection between thickness of sets and winning strategies of potential games has been studied by Yavicoli et al in \autocite{falconer2022intersections-ThicknessBasedGame,OriginalAlexiaPaper,yavicoli2022thickness-BoundaryBasedGame}.
We can replicate these results for affine thickness using the matrix potential game, introduced by the author, Mitchell and Samuel in \cite{howat2025matrixpotentialgamestructures}.
This results in the following theorems on patterns and intersections of thick sets.
Note that these results hold in the square metric.
\begin{nonotheorem}\label{nothm:patterns}
    Given $n \in \N$, consider the square metric on $\R^n$ and let $A$ be a diagonal ($n\!\times\!n$)-matrix with diagonal entries $\beta_{11},\dots,\beta_{nn} \in (0, 1/5)$, setting $\beta_{\max} = \max_j \beta_{jj}$.
    Let $C\subseteq B[0,1]$ be a compact subset with gaps $(G_k)_{k\in J}$ and $E$.
    Moreover, $\tau_A(C)\not\in\{ -\infty,\infty\}$.
    
    Let $\alpha_{\tau_A(C)}=\prod_{j=1}^n \beta_{jj}^{\tau_A(C)-1}$.
    If there exist $c,\delta\in(0,1)$ and $M \in \N$ such that
        \begin{align}\label{pattern_inequal_1}
            M\alpha_{\tau_A(C)}^c\leq \delta^{2}\left(1-\left(\prod_{j=1}^n\beta_{jj}\right)^{1-c}\right) \quad \text{and} \quad 
            3^{-n} \prod\limits_{j=1}^{n} (1-5\beta_{jj}^{\lfloor\delta(M^{1/c}\alpha_{\tau_A(C)})^{-1}\rfloor}) > 8^n (1+2^{2n+1})\delta.
        \end{align} 
    Then given a finite non-empty set $F$ with at most $M$ elements, $y\in\R^n$ and $\lambda\in (0,(1-\beta_{\max})/\text{diam}(F))$, there exists a non-empty set $X\subseteq \R^n$ such that $\lambda F+x\subseteq (C\cup E)\cap B[y,1]$ for all $x\in X$.
    Moreover, if
    \begin{align*}
        M\alpha_{\tau_A(C)}^c\leq \min\{\delta^2,K_{M}^{-1} n \log \beta_{\max}^{-1}\}\left(1-\left(\prod_{j=1}^n\beta_{jj}\right)^{1-c}\right),
    \end{align*}
    where 
        \begin{align*}
        K^{}_{M} =2\delta^{-1}\left\lvert \log\left( 3^{-n}\prod\limits_{j=1}^{n} (1-5\beta_{jj}^{\lfloor\delta(M^{\frac{1}{c}}\alpha_{\tau_A(C)})^{-1}\rfloor}) - 8^n (1+2^{2n+1})\delta\right)\right\rvert,
        \end{align*}
    then $\dim_{H}(X) \geq n-K^{}_{M} \alpha_{\tau_A(C)} (\lvert\log (\beta_{\max}) \rvert)^{-1}>0$.
\end{nonotheorem}

\begin{nonotheorem}\label{nocor:intersections}
    Given $n \in \N$, consider the square metric on $\R^n$ and let $A$ be a diagonal ($n\!\times\!n$)-matrix with diagonal entries $\beta_{11},\dots,\beta_{nn} \in (0, 1/5)$.
    Moreover, let $I$ be an at most countable set and for each $i \in I$, let $C_i\subseteq B[0,1]$ be a compact set with gaps $(G_k^{(i)})_{k\in J_i}$ and $E^{(i)}$.
    Moreover, for each $i\in I$, $\tau_A(C_{i})\not\in \{-\infty,\infty\}$.
    
    Set $C = \cap_{i \in I} (C_i\cup E^{(i)})$ and $\alpha_{\tau_A(C_i)}=\prod_{j=1}^n \beta_{jj}^{\tau_A(C_i)-1}$.
    If there exists $\alpha,c,\delta \in (0,1)$ such that $\alpha^c=\sum_{i\in I} \alpha_{\tau_A(C_i)}^c$ and 
        \begin{align}\label{pattern_inequal_2}
            \alpha^c\leq \delta^2\left(1-\left(\prod_{j=1}^n\beta_{jj}\right)^{1-c}\right)<1 \quad \text{and} \quad 0<\prod\limits_{j=1}^n\left( \frac{1}{3}(1-5\beta_{jj}^{\lfloor\delta\alpha^{-1}\rfloor})\right)-8^n (1+2^{2n+1})\delta,
        \end{align}
        then, for every $y\in\R^n$, we have that $C\cap \big(A(B[0,1])+y\big)\neq \emptyset$.
        Moreover,
    \begin{align*}
        \text{dim}_H\left(C\cap \big(A(B[0,1])+y\big)\right)\geq\max\left\{ n-K_1\frac{\alpha}{\left|\log(\beta_{\max})\right|},0\right\}.
    \end{align*}
\end{nonotheorem}

\section{An Affine Gap Lemma in $\R^n$}
Throughout this section we consider the square or Euclidean metric on $\R^n$.
Given two compact sets $C_1,C_2\subseteq \R$, the following result holds linking thickness and intersection, this is known as the Newhouse gap lemma.
\begin{thrm}[Newhouse Gap Lemma]
    Given two compact sets $C_1,C_2\subseteq \R$ such that neither set lies in a gap of the other, if $\tau(C_1)\tau(C_2)>1$ then $C_1\cap C_2\neq\emptyset$.
\end{thrm}
Falconer and Yavicoli claim in \autocite[Theorem 10]{falconer2022intersections-ThicknessBasedGame}, that this result extends to $\R^n$ directly, and one can obtain a Newhouse Gap Lemma in $\R^n$ with the same conditions using Falconer-Yavicoli thickness.
However, this is not the case.
In the following proposition we present an example of sets in $\R^n$ whose intersection is empty yet their affine thickness is sufficiently large and neither set lies in a gap of the other.
This leads to \Cref{cor:CounterexampletoTHRM10}, which directly contradicts Theorem 10 in \cite{falconer2022intersections-ThicknessBasedGame}.
\begin{prop}\label{prop: counter examp}
   Consider the square or Euclidean metric on $\R^n$ and $n\geq 2$.
   Let $A$ be a diagonal matrix with entries $\beta_{11},\dots,\beta_{nn}\in (0,1)$.
   There exists compact sets $C_1,C_2\subseteq B[0,1]$ such that neither of them is contained in a gap of the other, $\tau_A(C_1)+\tau_A(C_2)>0$
   and $C_1\cap C_2=\emptyset$.
\end{prop}
\begin{proof}
    We will construct and calculate the thickness of explicit examples in the square metric first.
    Let $r,s,t\in(0,1/4)$ such that $r>s>t$ to be determined.
    Let $x=(x_1,\dots,x_n)$ be such that $x_1=3/4$ and $x_i=0$ for $i\geq2$.
    We set 
    \begin{align*}
        C_1=\left(B[x,1/4]\setminus \open{B}[x,r]\right)\cup \{x\}\cup\left(B[-x,1/4]\setminus \open{B}[-x,r]\right)
    \text{ \, and \,}
        C_2=\left(B[x,s]\setminus \open{B}[x,t]\right)\cup \{-x\},
    \end{align*}
    these sets are illustrated below in $\R^2$ for the square metric.
    \begin{figure}[H]\label{pic: Counter examp}
    \centering
    \begin{tikzpicture}[x=0.625cm,y=0.625cm]

\draw (0,0) rectangle (8,8);
\draw (10,0) rectangle (18,8);
\draw (5,-10) rectangle (13,-2);

\draw [fill, cyan] (0,0+3) rectangle (2,2+3);
\draw [fill, white] (0.5,0.5+3) rectangle (1.5,1.5+3);
\draw [fill, cyan] (6,6-3) rectangle (8,8-3);
\draw [fill, white] (6.5,6.5-3) rectangle (7.5,7.5-3); 
\draw [fill,cyan] (7,7-3) circle (1.25pt);

\draw [fill, red] (16.625,6.625-3) rectangle (17.375,7.375-3);
\draw [fill, white] (16.845,6.845-3) rectangle (17.155,7.15-3);
\draw [fill,red] (11,1+3) circle (1.25pt);

\draw [fill, cyan] (5,-10+3) rectangle (7,-8+3);
\draw [fill, white] (5.5,-9.5+3) rectangle (6.5,1.5-10+3);
\draw [fill, cyan] (11,6-10-3) rectangle (13,8-10-3);
\draw [fill, white] (11.5,6.5-10-3) rectangle (12.5,7.5-10-3); 
\draw [fill, red] (16.625-5,6.625-10-3) rectangle (17.375-5,7.375-10-3);
\draw [fill, white] (16.845-5,6.845-10-3) rectangle (17.155-5,7.155-10-3);
\draw [fill,red] (11-5,1-10+3) circle (1.25pt);
\draw [fill,cyan] (12,7-10-3) circle (1.25pt);

\draw (4,-0.75) node {$C_1$};
\draw (14,-0.75) node {$C_2$};
\end{tikzpicture}
    \vspace{-1cm}
    \caption{The sets $C_1$ and $C_2$ both individually and together contained in $B[0,1]$ with circles denoting the singletons at $x$ and $-x$.}
    \end{figure}
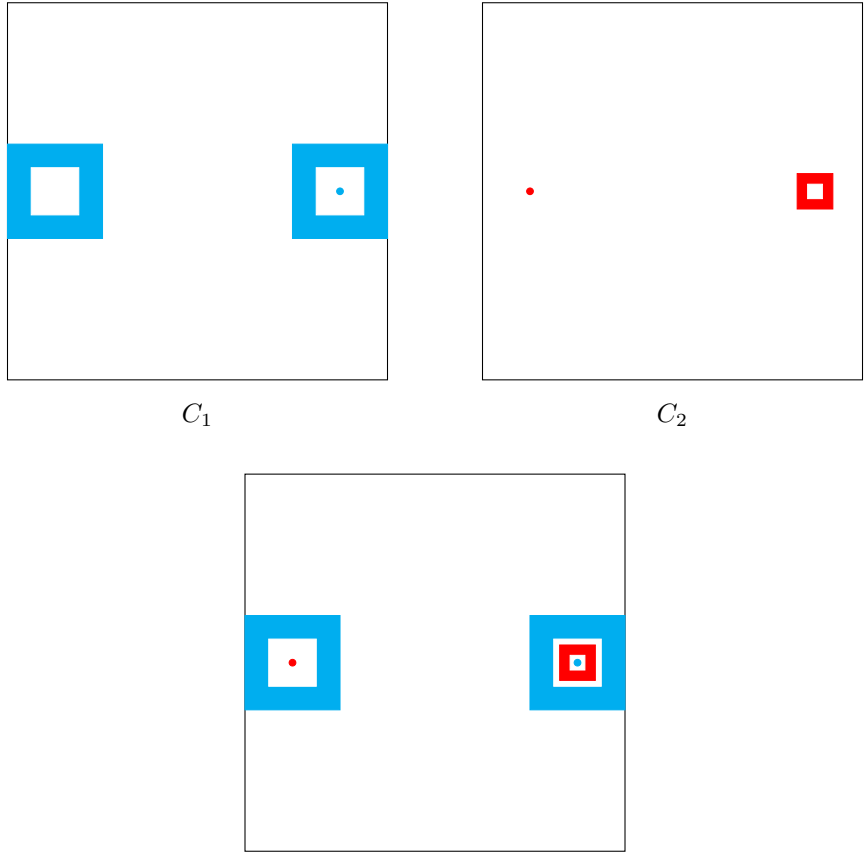
    It is clear that $C_1\cap C_2=\emptyset$ therefore, we require to find the thickness of $C_1$ and $C_2$ and show that neither of them is contained in a gap of the other.
    
    We can observe that the gaps of $C_1$ are $E^1=\R^n\setminus (B[x,1/4]\cup B[-x,1/4])$, $G^1_1=\open{B}[-x,r]$ and $G^1_2=\open{B}[x,r]\setminus \{x\}$ whereas the gaps of $C_2$ are $E^2=\R^n\setminus (B[x,s]\cup \{-x\})$ and $G^2_1=\open{B}[x,t]$.
    It is clear that $C_1\setminus (C_1\cap E^2)=\{x\}$ and $C_1\cap G^2_1=\{x\}$ hence $C_1$ is not contained in only one gap of $C_2$.
    Moreover, $C_2\cap G^1_1=\{-x\}$ and $C_2\cap G^2_1=B[x,s]\setminus \open{B}[x,t]$ hence $C_2$ is not contained in only one gap of $C_1$.

    We now find their size with respect to the matrix $A$.
    We can observe that $S_A(G_1^1)=S_A(G_2^1)$ satisfies $ \min_j\{\beta_{jj}\}^{1/S_A(G_1^1)}=r $ since this is when all the sides of the rectangle $A^{1/S_A(G_1^1)}(B[0,1])$ are larger than the sides of the gaps of $C_1$ and the smallest side is equal to one of the sides of the gaps.
    Hence,
    $$\frac{1}{S_A(G_1^1)}=\frac{1}{S_A(G_2^1)}=\log_{\min_j\{\beta_{jj}\}}(r)$$ and by similar reasoning $S_A(G_1^2)$ satisfies $$\frac{1}{S_A(G_1^2)}= \log_{\min\{\beta_{jj}\}}(t).$$
    Observe that the value $GD_A(1,C_1)=GD_A(2,C_1)$ satisfies  $$2\max\left\{\beta_{jj}\right\}^{1/GD_A(1,C_1)}=2\max\{\beta_{jj}\}^{1/GD_A(2,C_1)}= \frac{1}{4}-r$$ since $\frac{1}{4}-r$ is the smallest distance of the bounded gaps of $C_1$ from the unbounded gap of $C_1$ on each axis.
    Therefore, the largest side of the rectangle $A^{1/GD_A(1,C_1)}$ equals this distance.
    By a similar reasoning, $GD_A(1,C_2)$ satisfies $$\max_{j}\left\{\beta_{jj}\right\}^{1/GD_A(1,C_1)}= \frac{1}{2}(s-t).$$
    Hence $$\tau_A(C_1)=\log_{\min\{\beta_{jj}\}}(r)-\log_{\max\{\beta_{jj}\}} \left(\frac{1}{8}-\frac{r}{2}\right)  \text{ \, and \,} \tau_A(C_2)=\log_{\min\{\beta_{jj}\}}(t)-\log_{\max\{\beta_{jj}\}} \left(\frac{s}{2}-\frac{t}{2}\right).$$

    We now observe that for any $s$, we can make $t$ sufficiently small such that $\tau_A(C_2)> 0$ and moreover we can make $r$ sufficiently small such that $\tau_A(C_1)> 0$.
    Therefore, ensuring $r>s>t$, we can find values of $r,s,t$ such that $\tau_A(C_1)+\tau_A(C_2)>0$, and $C_1\cap C_2=\emptyset$.
    
    To complete the proof we simply note that this example in the square metric can be generalised for the Euclidean metric, however the calculations are not as pleasant.
\end{proof}
\begin{cor}\label{cor:CounterexampletoTHRM10}
    There exists sets $C_1,C_2\subseteq \R^n$ such that neither set lies in a gap of the other, $\tau(C_1)\tau(C_2)>1$ and $C_1\cap C_2\neq\emptyset$.
\end{cor}
\begin{proof}
    This follows by combination of Proposition~\ref{prop: counter examp} and \Cref{rem: our thickness is similar to Yavicoli-Falconer thickness}.
\end{proof}
This above proposition shows that in higher dimensions, a gap lemma using the natural cut out structure requires additional conditions.
A proof of an affine gap lemma in $\R^n$, under additional conditions, is the focus of the section.

\subsection{Gap lemma for BG linked pairs of sets}
We look to prove a gap lemma by adapting techniques presented by Falconer and Yavicoli in \cite{falconer2022intersections-ThicknessBasedGame}.
To do so we begin by making the following definition.
\begin{defin}\label{def: linked sets def}
    We say that gaps $U,V\subseteq \R^n$ are linked if $U\cap V\neq\emptyset$, $(\partial U)\setminus V\neq\emptyset$ and $(\partial V)\setminus
    U\neq\emptyset$.
    
    We say that compact sets $C_1$ and $C_2$ are BG linked if for every bounded gap $G^1$ of $C_1$ and bounded gap $G^2$ of $C_2$, $G^1$ and $G^2$ are linked or their intersection is empty.
\end{defin}
Sets that are BG linked are easier to handle in terms of their intersection properties, hence we study them first.
To do so we state and prove a useful fact about gap structures which is important to the following proofs.
\begin{lem}\label{lem: size go to zero or finite}
    Let $A$ be a diagonal matrix with entries $\beta_{11},\dots,\beta_{nn}\in (0,1)$.
    If $C\subseteq B[0,1]$ is a compact set with $\tau_A(C)\in (-\infty,\infty]$ then either there are finitely many gaps or $\lim\limits_{k\rightarrow\infty} S_A(G_k)= \lim\limits_{k\rightarrow\infty} \operatorname{diam}(G_k) =0$.
\end{lem}
\begin{proof}
    Assume for a contradiction that there exists a set $C$ with infinitely many gaps and for all $k\in\N$ we have that $S_A(G_k)\geq c$ for some $c>0$.
    For each $k\in\N$, take some $x_k\in G_k$ and observe that $ S_A(\{x_i,x_{k}\})\geq GD_A(k,C) $ for each $i<k$.
    We can observe that for each $k\in\N$ and $i<k$, since $GD_A(k,C)>0$, we have that 
    $$\tau_A(C)-\frac{1}{c}\leq \tau_{A}(C)-S_A(G_{k})^{-1}\leq -GD_A(k,C)^{-1}.$$
    Therefore, $0<GD_A(k,C)^{-1}\leq 1/c-\tau_A(C)$ and hence $$0<(1/c-\tau_A(C))^{-1}\leq GD_A(k,C)\leq S_A(\{x_i,x_{k}\}).$$
    We finally observe that if there exists a constant $M$ such that for each $k\in\N$ and each $i<k$ the value $S_A(\{x_i,x_{k}\})>M$, the sequence $(x_{k})_{k\in\N}$ has no convergent subsequences since $S_A(F)$ is linked to $\operatorname{diam}(F)$.
    This is a contradiction of the sequential compactness of $\R^n\setminus C$.
\end{proof}
The following is a technical result which ultimately leads to a contradiction in the proof of the BG linked affine gap lemma and has been isolated as a separate result to make the proof as concise and clear as possible.
\begin{lem}\label{lem: Linked Sets Sequence Exists}
    Let $A$ be a diagonal matrix with entries $\beta_{11},\dots,\beta_{nn}\in (0,1)$, and let $C_1,C_2\subseteq B[0,1]$ be non-empty compact sets such that $C_1,C_2$ are BG linked with $\tau_A(C_1)+\tau_A(C_2)>0$.
    Assume further that $C_1\cap C_2=\emptyset$ and that there exists bounded gaps $G^1$ of $C_1$ such that $\partial G^1\not\subseteq E^2$ and $G^2$ of $C_2$ such that $\partial G^2\not\subseteq E^2$.
    
    Under these assumptions, there exists a sequence of pairs $(G_{s_i}^{1},G_{t_i}^{2})_{i\in\N}$ of linked bounded gaps, such that for $i\in\N$, $t_i,s_i\in\N$, moreover, either $s_{i+1}=s_i$ and $t_{i+1}>t_i$ or $s_{i+1}>s_i$ and $t_{i+1}=t_i$.
    In addition, at least one of $\operatorname{diam}(G_{s_i}^{1})\rightarrow 0$ or $\operatorname{diam}(G_{t_i}^{2})\rightarrow 0$ as $i\rightarrow\infty$.
\end{lem}
\begin{proof}
    Let $(G_k^{1})_{k\in J_1}, E^{1}$ and $(G_k^{2})_{k\in J_2}, E^{2}$ be the gaps of $C_1$ and $C_2$ respectively.
    Let $\tau_1=\tau_A(C_1)$ and $\tau_2=\tau_A(C_2)$.
    Note that $\tau_1+\tau_2>0$ implies that neither take the value of $-\infty$.
    
    Observe that for $s\in J_1$ and $t\in J_2$ the following cannot hold simultaneously:
    \begin{equation}\label{eq: no both hold}
    GD_A(s,C_1)\leq S_A(G_t^{2}) \text{ \, and \, } GD_A(t,C_2)\leq S_A(G_s^{1}).
    \end{equation}
    Since, if both were to hold then
    $$GD_A(s,C_1)^{-1}\geq S_A(G_t^{2})^{-1} \text{ \, and \, } GD_A(t,C_2)^{-1}\geq S_A(G_s^{1})^{-1},$$
    which implies that
    $$ S_A(G_s^{1})^{-1}-\tau_1\geq S_A(G_t^{2})^{-1} \text{ \, and \, }  S_A(G_t^{2})^{-1}-\tau_2\geq S_A(G_s^{1})^{-1}. $$
    Combining these inequalities we would have that
    $$ S_A(G_s^{1})^{-1}-\tau_1-\tau_2\geq S_A(G_s^{1})^{-1},$$
    which would be a contradiction since $\tau_1+\tau_2>0$. 

    We will now construct a sequence of pairs $(G_{s_i}^{1},G_{t_i}^{2})_{i\in\N}$ of bounded linked gaps through induction, such that for $i\in\N$, $t_i,s_i\in\N$, moreover, either $s_{i+1}=s_i$ and $t_{i+1}>t_i$ or $s_{i+1}>s_i$ and $t_{i+1}=t_i$.

    Firstly we will define $(G^{1}_{s_1},G^{2}_{t_1})$.
    Assume for a contradiction that there are no pairs of linked bounded gaps and hence, since $C_1$ and $C_2$ are BG linked, we have that for each pair of bounded gaps $G^1$ of $C_1$ and $G^2$ of $C_2$ their intersection is empty.
    Consider the gap $G^1$ of $C_1$ with $\partial G^1\not\subseteq E^2$, ensured by our assumptions, since $C_1\cap C_2=\emptyset$ we have that
    \begin{align*}
        \emptyset\neq\partial G^1\setminus E^2\subseteq \bigcup_{j\in J_2} G_j^2.
    \end{align*}
    Now for some $x\in \partial G^1\setminus E^2$, since $\bigcup_{j\in J_2} G_j^2$ is open, there exists $r>0$ such that $\open{B}[x,r]\subseteq \bigcup_{j\in J_2} G_j^2$ moreover, by definition of the boundary $\open{B}[x,r]\cap G^1 \neq\emptyset$.
    Therefore 
    \begin{align*}
        G^1\cap \bigcup_{j\in J_2} G_j^2\neq \emptyset.
    \end{align*}
    a contradiction as $G^1\cap G_j^2=\emptyset$ for all $j\in J_2$.
    Hence, there exist a pair of linked bounded gaps.
    Denote this pair $(G^1_{s_1},G^2_{t_1})$.

    Let $k\in\N$ be given and let $(G^1_{s_k},G_{t_k}^2)$ be a pair of linked bounded gaps of $C_1$ and $C_2$ defined inductively.
    Since $(G^1_{s_k},G_{t_k}^2)$ are linked gaps there exists some $x\in \partial G^1_{s_k}\setminus G_{t_k}^2$ and since $x\in\partial G^1_{s_k}$ implies that $x\in C_1$ there exists a gap $V$ of $C_2$ such that $x\in V$.
    The sets $(G^1_{s_k},V)$ intersect.
    Through the same argument there exists $y\in\partial G_{t_k}^2\setminus G_{s_k}^1$ and gap $U$ of $C_1$ such that $y\in U$.
    Again, the sets $(U, G_{t_k}^2)$ intersect.
    
    If $GD_A(s_{k},C_1) > S_A(G_{t_k}^{2})$ we note that since $(G^1_{s_{k}},G_{t_{k}}^2)$ are linked sets and $GD_A(s_{k},C_1) > S_A(G_{t_k}^{2})$, we have that the closure of $G_{t_{k}}^2$ does not intersect $E_1$ or $G_{i}^{1}$ for any $i = 1,\dots, s_k-1$.
    Therefore, we have that $y\in U$ defined previously is such that $y\in \partial G^{2}_{t_{k}}\setminus G^{1}_{s_{k}}\subseteq (\bigcup_{i=1}^{s_k} G_i^{1}\cup E_1)^C\cap C_1^C$ and hence is a bounded gap. 
    In this case we set $(G^1_{s_{k+1}},G_{t_{k+1}}^2) = (U, G_{t_{k}}^2)$ and observe that $s_{k+1}>s_k$ and $t_{k+1}=t_k$ as required, moreover since the sets intersect and are bounded they are linked.
     
    If $GD_A(s_k,C_1) \leq S_A(G_{t_k}^{2})$ we observe that  by \eqref{eq: no both hold} we have that $GD_A(t_k,C_2) > S_A(G_{s_k}^{1})$ and hence by an argument analogous to the above we have that $V$ is a bounded gap. 
    In this case we set $(G^1_{s_{k+1}},G_{t_{k+1}}^2) = (G_{s_k}^1, V)$ and observe that $s_{k+1}=s_k$ and $t_{k+1}>t_k$ as required, moreover since the sets intersect and are bounded they are linked.

    Now we have defined the sequence $(G_{s_i}^{1},G_{t_i}^{2})_{i\in\N}$ we verify that at least one of $\operatorname{diam}(G_{s_i}^{1})\rightarrow 0$ or $\operatorname{diam}(G_{t_i}^{2})\rightarrow 0$.
    Note that the defined sequence implies that at least one of the gaps sets is countable.
    The claim is verified in the following two cases:

    \underline{Case (I): $J_1$ and $J_2$ are countable.}
    In this case, we have that by Lemma~\ref{lem: size go to zero or finite}, $\operatorname{diam} (G_{k}^1)\rightarrow 0$ and $\operatorname{diam} (G_{k}^2)\rightarrow 0$ as $k\rightarrow\infty$. 
    Moreover, by construction we have that at least one of $(s_k)_{k\in\N}\subseteq \N$ or $(t_k)_{k\in\N}\subseteq \N$ increase infinitely often and hence at least one of $\operatorname{diam}(G^1_{s_k})$ and $\operatorname{diam}(G^2_{t_k})$ tends to zero.

    \underline{Case (II): One of the sets $J_1,J_2$ is finite.}
    We have that $(s_k)\subseteq J_1$, $(t_k)\subseteq J_2$ and by construction either $s_{i+1}=s_i$ and $t_{i+1}>t_i$ or $s_{i+1}>s_i$ and $t_{i+1}=t_i$.
    This combined with the fact that one of $J_1$ or $J_2$ is finite implies that there exists some large $K\in\N$ such that  $s_{k+1}=s_K$ and $t_{k+1}>t_k$ for each $k\geq K$ or $s_{k+1}>s_k$ and $t_{k+1}=t_K$ for each $k\geq K$.
    Hence, one of $(s_k)_{k\in\N}\subseteq \N$ or $(t_k)_{k\in\N}\subseteq \N$ increase infinitely often and hence, by Lemma~\ref{lem: size go to zero or finite}, one of $\operatorname{diam}(G^1_{s_k})$ and $\operatorname{diam}(G^2_{t_k})$ tends to zero.    
\end{proof}
\begin{prop}[BG linked affine gap lemma]\label{thrm: BG linked Affine Gap Lemma}
    Let $A$ be a diagonal matrix with entries $\beta_{11},\dots,\beta_{nn}\in (0,1)$ and let $C_1,C_2\subseteq B[0,1]$ be a BG linked pair of sets such that there exists bounded gaps $G^1$ of $C_1$ such that $\partial G^1\not\subseteq E^2$ and $G^2$ of $C_2$ such that $\partial G^2\not\subseteq E^1$. If $\tau_A(C_1)+\tau_A(C_2)>0$, then $C_1\cap C_2\neq\emptyset$.
\end{prop}
\begin{proof}
    Assume for a contradiction that $C_1\cap C_2=\emptyset$.
    We satisfy the conditions of \Cref{lem: Linked Sets Sequence Exists} and hence there exists a sequence of pairs $(G_{s_i}^{1},G_{t_i}^{2})_{i\in\N}$ of linked bounded gaps, such that for $i\in\N$, $t_i,s_i\in\N$, moreover, either $s_{i+1}=s_i$ and $t_{i+1}>t_i$ or $s_{i+1}>s_i$ and $t_{i+1}=t_i$.
    In addition, at least one of $\operatorname{diam}(G_{s_i}^{1})\rightarrow 0$ or $\operatorname{diam}(G_{t_i}^{2})\rightarrow 0$ as $i\rightarrow\infty$.

    Assume without loss of generality that $\operatorname{diam}(G_{s_i}^{1})\rightarrow 0$.
    Taking $x_k\in\partial G^1_{s_k}\subseteq C_1$ and $y_k\in G^1_{s_k}\cap \partial G^2_{t_k}\subseteq C_2$, we have that $d(x_k,y_k)\leq \operatorname{diam}(\tilde{G}^1_{t_k})$ and hence
    \begin{align}\label{eq: somewhat Cauchy stuff}
    \lim_{k\rightarrow\infty} d(x_k,y_k)= \lim_{k\rightarrow\infty}\operatorname{diam}(\tilde{G}^1_{t_k})=0.
    \end{align}
    Note that the point $y_k$ exists since if $G^1_{s_k}\cap \partial G^2_{t_k}=\emptyset$ then $\overline{G^1_{s_k}}\cap \partial G^2_{t_k}=\emptyset$ since $\partial G^1_{s_k}\cap \partial G^2_{t_k}=\emptyset$.
    Hence for $z\in \overline{G^1_{s_k}}$ we have that $z\in G^2_{t_k}$ or $z\notin \overline{G^2_{t_k}}$.
    Therefore, since $G^1_{s_k}\cap G^2_{t_k}\neq\emptyset$, we have that \begin{align}\label{clopenanddat}
    \overline{G^1_{s_k}}\supsetneqq\overline{G^1_{s_k}}\setminus G^2_{t_k}=\overline{G^1_{s_k}}\setminus ( \overline{G^1_{s_k}}\cap G^2_{t_k})=\overline{G^1_{s_k}}\setminus \overline{G^2_{t_k}}=\overline{G^1_{s_k}}\setminus (\overline{G^1_{s_k}}\cap\overline{G^2_{t_k}}).\end{align}
    We can observe that since $G_{s_k}^1$ is path-connected and therefore connected, it's closure $\overline{G^1_{s_k}}$ is connected and hence it's only clopen sets in the subspace topology are itself and the empty-set.
    We can see by \eqref{clopenanddat} that $\overline{G^1_{s_k}}\setminus G^2_{t_k}$ is clopen and not $\overline{G^1_{s_k}}$ and hence $\overline{G^1_{s_k}}\setminus G^2_{t_k}=\emptyset$, therefore $\overline{G^1_{s_k}}\subseteq G^2_{t_k}$, a contradiction as they are linked sets. 
    
    Therefore, since $(x_k)_{k\in\N}\subseteq \tilde{C}_1$ there exists a convergent subsequence $(x_{k_j})_{j\in\N}$ such that $x_{k_j}\rightarrow x\in \tilde{C}_1$ as $j\rightarrow \infty$.
    Moreover, by Equation~\eqref{eq: somewhat Cauchy stuff} the subsequence $(y_{k_j})_{j\in\N}\subseteq \tilde{C}_2$ is such that $y_{k_j}\rightarrow x$ and hence $x\in \tilde{C}_2$.
    Therefore, $\tilde{C}_1\cap \tilde{C}_2\neq \emptyset$ and hence we have a contradiction.

    To conclude, assuming that $C_1\cap C_2=\emptyset$ alongside the conditions of the Theorem, implies that $C_1\cap C\neq\emptyset$.
    This is a contradiction and hence $C_1\cap C_2\neq\emptyset$.
\end{proof}
The BG gap lemma is a strong tool, however the condition of linked bounded gaps is reasonably restrictive.
In the unit interval it is possible to convert a pair of compact sets $C_1$ and $C_2$ into a pair of BG linked sets $\tilde{C}_1,\tilde{C_2}$ without decreasing thickness and having $C_1\cap C_2=\tilde{C}_1\cap \tilde{C}_2$.
This is the approach to prove the Newhouse gap lemma.
Unfortunately in higher dimensions more consideration is needed, as shown by \Cref{prop: counter examp}.
This motivates the following definition.

\begin{defin}\label{def: strongrefinable}
     Let $A$ be a diagonal matrix with entries $\beta_{11},\dots,\beta_{nn}\in (0,1)$.
     We say that a pair of compact sets $C_1,C_2\subset B[0,1]$ is a strongly refinable pair for the matrix $A$ if there exists compact sets $\tilde{C}_1,\tilde{C}_2\subseteq B[0,1]$ such that the following hold:
    \begin{enumerate}
        \item $\tilde{C}_1$ and $\tilde{C}_2$ are BG linked sets.
        \item There exists bounded gaps $\tilde{G}^1$ of $\tilde{C}_1$ such that $\partial \tilde{G}^1\not\subseteq \tilde{E}^2$ and $\tilde{G}^2$ of $\tilde{C}_2$ such that $\partial \tilde{G}^2\not\subseteq \tilde{E}^1$.
        \item $\tau_A(\tilde{C}_1)+\tau_A(\tilde{C}_2)\geq \tau_A(C_1)+\tau_A(C_2)$.
        \item $\tilde{C}_1\cap \tilde{C_2}\subseteq C_1\cap C_2$.
    \end{enumerate}
\end{defin}
Hence we can restate \Cref{thrm: BG linked Affine Gap Lemma} to obtain the gap lemma stated in the introduction.
\begin{thrm}[Affine Gap Lemma]\label{thrm: Affine Gap Lemma}
    Let $A$ be a diagonal matrix with entries $\beta_{11},\dots,\beta_{nn}\in (0,1)$ and let $C_1,C_2\subseteq B[0,1]$ be a strongly refinable pair, non-empty and compact. If $\tau_A(C_1)+\tau_A(C_2)>0$, then $C_1\cap C_2\neq\emptyset$.
\end{thrm}
\begin{proof}
    Since $C_1$ and $C_2$ are a strongly refinable pair, we have that there exists compact sets $\tilde{C}_1$, $\tilde{C}_2$ satisfying the conditions of \Cref{thrm: BG linked Affine Gap Lemma}.
    Hence $\emptyset\neq\tilde{C}_1\cap\tilde{C}_2\subseteq C_1\cap C_2$.
\end{proof}

\subsection{A Condition for strong refinability}
We now look to determine a reasonable class of strongly refinable pairs of sets, in which the affine gap lemma can be applied.
We note that this class should be easily verifiable, indeed this ease of use is one of the strengths of gap lemmas.
The following lemma is the first step towards this class, proving that when gaps are contained in other gaps we can remove them whilst keeping desirable conditions.
\begin{lem}\label{lem: contained gaps are fine}
    Let $A$ be a diagonal matrix with entries $\beta_{11},\dots,\beta_{nn}\in (0,1)$ and let $C_1,C_2\subseteq B[0,1]$ be non-empty compact sets.
    There exists compact sets $C_1'$ and $C_2'$ such that the following hold:
    \begin{itemize}
        \item $C_1\cap C_2=C_1'\cap C_2'$;
        \item $\tau_A(C_1')+\tau_A(C_2')\geq \tau_A(C_1)+\tau_A(C_2)$;
        \item No bounded gap ${G^1}'$ of $C_1'$ is such that $\overline{{G^1}'}\subseteq {G^2}'$ for any bounded gap ${G^2}'$ of $C_2'$ and no bounded gap ${G^2}'$ of $C_2'$ is such that $\overline{{G^2}'}\subseteq {G^1}'$ for any bounded gap ${G^1}'$ of $C_1'$.
    \end{itemize}
\end{lem}
\begin{proof}
    Consider the sets $C_1$ with gaps ${E^1}$, $({G^1}_k)_{k\in J_1}$ and $C_2$ with gaps ${E^2}$, $({G^2}_k)_{k\in J_2}$.
    We set
        \begin{align*}
            C_1'=C_1\cup \bigcup_{j\in J_2}\bigcup_{\substack{i\in J_1 : \\\overline{G^1_{i}}\subseteq {G^2_{j}}}} {G_i^1}.
        \end{align*}
        It is clear that $C_1'$ is a compact subset of $B[0,1]$ and $C_1'\cap C_2=C_1\cap C_2$.
        Moreover, no bounded gap ${G^1}'$ of $C_1'$ is such that $\overline{{G^1}'}\subseteq {G^2}'$ for any bounded gap ${G^2}$ of $C_2$ since the gaps of $C_1$ satisfying this condition have been removed.
        Moreover, since we are only removing gaps we have that $\tau_A(C_1')\geq \tau_A(C_1)$.
        
        Consider the sets $C_1'$ with gaps ${E^1}'$, $({G^1}_k')_{k\in J_1'}$ and $C_2$ with gaps ${E^2}$, $({G^2}_k)_{k\in J_2}$.
        We set
        \begin{align*}
            C_2'=C_2\cup \bigcup_{j\in J_1'}\bigcup_{\substack{i\in J_2 : \\\overline{G^2_{i}}\subseteq {G^1_{j}}'}} {G_i^2}.
        \end{align*}
        Again it is clear that $C_1''$ is a compact subset of $B[0,1]$ and $C_1'\cap C_2'=C_1\cap C_2$.
        Moreover, no bounded gap ${G^2}'$ of $C_2'$ is such that $\overline{{G^2}'}\subseteq {G^1}'$ for any bounded gap ${G^1}$ of $C_1'$ since the gaps of $C_2'$ satisfying this condition have been removed.
        Moreover, since we are only removing gaps we have that $\tau_A(C_2')\geq \tau_A(C_2)$.
\end{proof}
This lemma can be used to obtain the following useful class of strongly refinable pairs.
\begin{prop}\label{prop:sufficient cond for strong ref}
    Let $C_1,C_2\subset B[0,1]$ be a pair of compact sets such that the following hold:
    \begin{itemize}
        \item There exists bounded gaps $G^1$ of $C^1$ and $G^2$ of $C_2$ such that $\overline{G^1}$ is not contained in a gap of $C_2$ and $\overline{G^2}$ is not contained in a gap of $C_1$; moreover $\partial G^1\not\subseteq E^2$ and $\partial G^2\not\subseteq E^1$.
        \item For any pair of bounded gaps $G^1$ of $C_1$ and $G^2$ of $C_2$ we have that if $\partial G^1\subseteq G^2$ then $G^1\subseteq G^2$ and if $\partial G^2\subseteq G^1$ then $G^2\subseteq G^1$.
    \end{itemize}
    Under these conditions $C_1$ and $C_2$ are a strong refinable pair.
\end{prop}
\begin{rem}
    We make the remark that the second condition is satisfied when the gaps are simply connected, this is useful for a large class of cookie cutter sets, including certain Bedford-McMullen carpets.
\end{rem}
\begin{proof}
    Observe that $C_1'$ and $C_2'$ constructed in \Cref{lem: contained gaps are fine} clearly satisfy condition (3) and (4) of \Cref{def: strongrefinable}.
    Moreover, by the second condition in the proposition, for bounded gaps $G^1$ and $G^2$ if $\partial G^1\subseteq G^2$ we have that $\overline{G^1}\subseteq G^2$ and if $\partial G^2\subseteq G^1$ we have that $\overline{G^2}\subseteq G^1$.
    Therefore $C_1'$ and $C_2'$ are BG linked as any gap from $C_1$ and $C_2$ with boundary in another bounded gap is removed.

    We conclude by observing that our first condition implies that there exists gaps satisfying condition (2) in the sets $C_1$ and $C_2$ which are not removed in the construction of $C_1'$ and $C_2'$ since they are not contained in a bounded gap. 
    Hence the sets $C_1$ and $C_2$ are strongly refinable.
    \end{proof}

We finish this section by presenting some small questions about the gap lemma in higher dimensions.
\begin{itemize}
    \item \Cref{prop:sufficient cond for strong ref} presents sufficient conditions for a pair of compact sets to be strongly refinable. How far can the conditions be weakened?
    \item Does the class of sets not satisfying \Cref{prop:sufficient cond for strong ref} contain well known examples of fractal sets?
\end{itemize}

\section{Affine thickness and the matrix potential game}
The potential game was introduced in \cite{MR3826896} and a quantitive variant was defined in \cite{broderick2017quantitative-OriginalRestrictedPotential}.
In \cite{howat2025matrixpotentialgamestructures}, the author, Mitchell and Samuel defined a variant of the potential game, called the matrix potential game, to study patterns and intersections of self-affine sets, this being defined below.
This work complemented and extended the work of Yavicoli et al in \cite{falconer2022intersections-ThicknessBasedGame,OriginalAlexiaPaper,yavicoli2022thickness-BoundaryBasedGame}, in which variants of Newhouse thickness were introduced, and it was shown that certain sets with positive thickness are winning in the potential game.
In this section the author provides complementary results for thickness with respect to an affine matrix $A$.

Note that throughout this section we exclusively consider the square metric.

\begin{defin}\label{def:winning}
    Let $A$ be a diagonal $(n\!\times\!n)$-matrix with diagonal entries $\beta_{11},\dots, \beta_{nn}\in (0,1)$.
    Given $\rho_1\geq \rho_2>0$, $\alpha>0$ and $c\in[ 0,1)$, the $(\alpha,A,c,\rho_2,\rho_1)$-potential game is played with the following rules, with Player\;II having the option to skip their turn:
    \begin{itemize}
        \item On their first turn, Player\;I plays a closed set $U_1=A(B[0,r])+b_1$ satisfying $\rho_1\geq r\geq \rho_2$.
        \item On their first turn, Player\;II responds by selecting a collection of tuples
        \begin{align*}
            \mathcal{A}(U_1) = \left\{(q_{i,1},y_{i,1}) :  q_{i,1} \in \R \; \text{with} \; q_{i,1} \geq 1, y_{i,1}\in\R^n \; \text{and} \; i\in I_1\right\},
        \end{align*}
        where $I_1$ is an most countable index set.
        This defines the deleted set 
        \begin{align*}
        \bigcup_{\substack{(q_{i,1},y_{i,1})\\\in\mathcal{A}(U_1)}} \left(A^{q_{i,1}}(B[0,r])+y_{i,1}\right),
        \end{align*}
        which will be used to define the winning conditions at the end of the game. If $c>0$, then Player\;II's collection must satisfy
            \begin{align*}
            \sum\limits_{i\in I_1}\prod_{j=1}^n\left( \beta_{jj}^{q_{i,1}}\right)^c\leq \left(\alpha\prod_{j=1}^n\beta_{jj}\right)^{c}\hspace{-0.625em}.
            \end{align*}
        If $c=0$, then Player\;II can only choose one pair $(q_{1,1},y_{1,1})$ which satisfies:
            \begin{align*}
            \prod_{j=1}^n\beta_{jj}^{q_{1,1}}\leq\alpha \prod_{j=1}^n\beta_{jj}.
            \end{align*}
        \item For each $m\in\N\setminus\{1\}$, on their $m$-th turn, Player\;I plays a closed set $U_m=A^{m}\left(B[0,r]\right)+b_m$ such that $U_m\subseteq U_{m-1}$.
        \item On their $m$-th turn, Player\;II responds by selecting a collection of tuples 
        \begin{align*}
            \mathcal{A}(U_m,\dots,U_1) = \left\{(q_{i,m},y_{i,m}): q_{i,m} \in \R \; \text{with} \; q_{i,m} \geq 1,  y_{i,m}\in\R^n \; \text{and} \; i\in I_m\right\},
        \end{align*}
        where $I_m$ is an at most countable index set. 
        This defines the deleted set
        \begin{align*}
            \bigcup_{\substack{(q_{i,m},y_{i,m})\\\in\mathcal{A}(U_m,\dots,U_1)}} \left(A^{q_{i,m}}(B[0,r])+y_{i,m}\right).
        \end{align*}
        If $c>0$ ,then Player\;II's collection must satisfy
        \begin{align*}
        \sum\limits_{i\in I_m}\prod_{j=1}^n\left( \beta_{jj}^{q_{i,m}}\right)^c\leq \left(\alpha\prod_{j=1}^n\beta_{jj}^{m}\right)^{c}\hspace{-0.625em}. 
        \end{align*}
        If $c=0$, then Player\;II can only choose one pair $(q_{1,m},y_{1,m})$ which satisfies:
            \begin{align*}
            \prod_{j=1}^n\beta_{jj}^{q_{1,m}}\leq\alpha \prod_{j=1}^n\beta_{jj}^{m}.
            \end{align*}
    \end{itemize}
\end{defin}
By construction, there exists a single point $x_{\infty}$, with $\{x_{\infty}\} = \cap_{m\in\N}U_m$, called the outcome of the game and a deleted region,
    \begin{align}\label{eq:winning-condition-deletion}
        \operatorname{Del}(\mathcal{A}(U_1),\mathcal{A}(U_1,U_2),\dots) = \bigcup_{m\in\N} \bigcup_{\substack{(q_{i,m},y_{i,m})\\\in\mathcal{A}(U_m,\dots,U_1)}} \left(A^{q_{i,m}}(B[0,r])+y_{i,m}\right).
    \end{align}
We say that a set $S\subseteq \R^n$ is an \textit{$(\alpha,A,c,\rho_2,\rho_1)$-winning set}, if Player\;II has a strategy ensuring that if $x_{\infty}\not\in \operatorname{Del}(\mathcal{A}(U_1),\mathcal{A}(U_1,U_2),\dots)$, then $x_{\infty}\in S$.

The following lemmas \autocite[Lemma 3.3 and Lemma 3.4]{howat2025matrixpotentialgamestructures} are properties of the matrix potential game:
\begin{lem}[Monotonicity]\label{lem:monotonicity}
    Let $A$ be a diagonal $(n\!\times\!n)$-matrix with diagonal entries $\beta_{11},\dots, \beta_{nn}\in(0,1)$, $\rho_1 \geq \rho_2 > 0$, $\alpha \in (0,1]$ and $c\in[ 0,1)$. If $S$ is an $(\alpha,A,c,\rho_2,\rho_1)$-winning set, then for $\alpha',c',\rho_2',\rho_1'$ satisfying $\alpha\leq \alpha'$, $c'\geq c$, $\rho_2\leq\rho_2'\leq\rho_1'\leq \rho_1$, the set $S$ is $(\alpha',A,c',\rho_2',\rho_1')$-winning.
\end{lem}

\begin{lem}[Countable intersections]\label{lem:countableintersections}
    Let $A$ be a diagonal $(n\!\times\!n)$-matrix with diagonal entries $\beta_{11},\dots, \beta_{nn}\in(0,1)$.
    Let $\{S_j\}_{j\in J}$ be an at most countable collection of sets such that for each $j\in J$, the set $S_j$ is a $(\alpha_j,A,c,\rho_2,\rho_1)$-winning set with $c>0$. If $\alpha = (\sum_{j\in J}(\alpha_j)^c)^{1/c}$ is finite, then $S= \cap_{j\in J} S_j$ is an $(\alpha,A,c,\rho_2,\rho_1)$-winning set.
\end{lem}
We now prove that thick sets are winning in the matrix potential game, allowing us to apply the results of \cite{howat2025matrixpotentialgamestructures} to thick sets.
\begin{prop}\label{prop: Winning conditions}
    Let $A$ be a diagonal matrix with entries $\beta_{11},\dots,\beta_{nn}\in (0,1)$, and let $C\subseteq B[0,1]$ be a compact set with gaps $(G_k)_{k\in J}$ and $E$.
    If $\tau_A(C)\not\in \{ -\infty,\infty\}$, then $C\cup E$ is $(\alpha_{\tau_A(C)},A,0,1,1)$-winning in the matrix potential game, where $\alpha_{\tau_A(C)}=\prod_{j=1}^n \beta_{jj}^{\tau_A(C)-1}$.
\end{prop}
\begin{rem}
    We exclude the case where $\tau_A(C)=\infty$ since this implies that $E$ is the only connected component, and hence $E\cup C=\R^n$.
    Therefore, there is no need to consider these sets.
\end{rem}
\begin{proof}
    We define a general rule and prove that this is a winning strategy for Player II.
    Let $U_m=A^{m}(B[0,r])+b_m$ be the $m^{\text{th}}$ play of the game by Player I, noting that $r\in [\rho_2,\rho_1]$.
    If there exists $i\in J$ such that $U_m$ intersects $G_i$ and $k>GD_A(i,C)^{-1}$ we have that $U_m\cap G_k=\emptyset$ for all $k<i$ and $U_m\cap E=\emptyset.$
    Therefore, if it is a legal play, Player II deletes the collection $\mathcal{A}(U_m)=\{(S_A(G_i)^{-1},z)\}$ where $z$ is such that $A^{S_A(G_i)^{-1}}(B[0,1])+z\supset G_i$; the existence of which is guaranteed by the definition of $S_A(G_i)$.
    Otherwise Player II skips their turn.

    To verify Player II's winning strategy assume that $x_{\infty}\not \in \operatorname{DEL}(\mathcal{A}(U_1),\dots)$ and for a contradiction that $x_{\infty}\not \in C\cup E$ and hence there exists a gap $G_i$ with $x_{\infty}\in G_i$.
    Observe, that since $x_{\infty}\in U_m$ for all $m\in\N$ and $\tau_A(C)\neq -\infty$, there exists $N\in\N$ such that $k> GD_A(i,C)^{-1}$ for all $k\in\N_0$ with $k\geq N$ and $k\leq GD_A(i,C)^{-1}$ for all $k\in\N_0$ with $k<N$.
    Hence, on the $N^{\text{th}}$ turn of the game $U_N$ intersects $G_i$ and $ N>GD_A(i,C)^{-1}\geq N-1$, however the set $G_i$ is not deleted, therefore the play $\mathcal{A}(U_N)=\{(S_A(G_i)^{-1},z)\}$ where $z$ is such that $A^{S_A(G_i)^{-1}}(B[0,1])+z\supset G_i$ is not legal, by Player II's strategy.

    However, we can observe that
    \begin{align*}
        \prod_{j=1}^n \beta_{jj}^{S_A(G_i)^{-1}}&= \prod_{j=1}^n \left(\beta_{jj}^{S_A(G_i)^{-1}-GD_A(i,C)^{-1}+GD_A(i,C)^{-1}}\right)\\
        &\leq \prod_{j=1}^n \left(\beta_{jj}^{\tau_A(C)+GD_A(i,C)^{-1}}\right)\leq\prod_{j=1}^n\left(\beta_{jj}^{\tau_A(C)+N-1}\right) =\alpha \prod_{j=1}^n \beta_{jj}^{N}.
    \end{align*}
    Hence, $\mathcal{A}(U_m)$ is a legal play and Player II's strategy dictates its deletion.
    This is a contradiction. Therefore, $x_{\infty}\in C\cup E$ and hence Player II's strategy is winning.
\end{proof}
\section{Intersections and pattern properties in sets with large thickness}
In \cite{howat2025matrixpotentialgamestructures} the author, Mitchell and Samuel provide dimensional estimates for winning sets in the matrix potential game for the square metric $d_{\infty}(x,y)= \max\{|x_1-y_1|,\dots, |x_n-y_n|\}$.
Using this estimate, they provide interesting structural results about winning sets which thick sets are a subclass.
Therefore, we can restate these results for thick sets.
Note that in this section we must restrict to the square metric and all of the following results are stated in this metric.

We say that a set $S$ \textit{contains a homothetic copy of} $C=\{b_1,...,b_M\}$, with $M \in \N$, often called a \textit{pattern}, if there exists $\lambda >0$ such that $\cap_{i=1}^M (S-\lambda b_i)\neq\emptyset$.
Hausdorff dimension alone cannot prove the existence of patterns; indeed, there exist examples of full Hausdorff dimension sets that avoid patterns~\cite{MissingAllrectanglesKeleti,MissingcountablymeanytrianglessKeleti}.
The existence of patterns in families of sets has been well studied, see for example~\cite{broderick2017quantitative-OriginalRestrictedPotential,pattsparcesets,falconer2022intersections-ThicknessBasedGame,howat2025matrixpotentialgamestructures,Arithmetic_progressions_in_sets_of_fractional_dimension,Smallsetscontaining,OriginalAlexiaPaper,yavicoli2022thickness-BoundaryBasedGame}.
We extend this study to discuss the patterns in sets with large affine thickness.
\begin{thrm}\label{thm:patterns}
    Given $n \in \N$, consider the square metric on $\R^n$ and let $A$ be a diagonal ($n\!\times\!n$)-matrix with diagonal entries $\beta_{11},\dots,\beta_{nn} \in (0, 1/5)$, setting $\beta_{\max} = \max_j \beta_{jj}$.
    Let $C\subseteq B[0,1]$ be a compact subset with gaps $(G_k)_{k\in J}$ and $E$.
    Moreover, $\tau_A(C)\not\in\{ -\infty,\infty\}$.
    
    Let $\alpha_{\tau_A(C)}=\prod_{j=1}^n \beta_{jj}^{\tau_A(C)-1}$.
    If there exist $c,\delta\in(0,1)$ and $M \in \N$ such that
        \begin{align}\label{pattern_inequal_1}
            M\alpha_{\tau_A(C)}^c\leq \delta^{2}\left(1-\left(\prod_{j=1}^n\beta_{jj}\right)^{1-c}\right) \quad \text{and} \quad 
            3^{-n} \prod\limits_{j=1}^{n} (1-5\beta_{jj}^{\lfloor\delta(M^{1/c}\alpha_{\tau_A(C)})^{-1}\rfloor}) > 8^n (1+2^{2n+1})\delta.
        \end{align} 
    Then given a finite non-empty set $F$ with at most $M$ elements, $y\in\R^n$ and $\lambda\in (0,(1-\beta_{\max})/\text{diam}(F))$, there exists a non-empty set $X\subseteq \R^n$ such that $\lambda F+x\subseteq (C\cup E)\cap B[y,1]$ for all $x\in X$.
    Moreover, if
    \begin{align*}
        M\alpha_{\tau_A(C)}^c\leq \min\{\delta^2,K_{M}^{-1} n \log \beta_{\max}^{-1}\}\left(1-\left(\prod_{j=1}^n\beta_{jj}\right)^{1-c}\right),
    \end{align*}
    where 
        \begin{align*}
        K^{}_{M} =2\delta^{-1}\left\lvert \log\left( 3^{-n}\prod\limits_{j=1}^{n} (1-5\beta_{jj}^{\lfloor\delta(M^{\frac{1}{c}}\alpha_{\tau_A(C)})^{-1}\rfloor}) - 8^n (1+2^{2n+1})\delta\right)\right\rvert,
        \end{align*}
    then $\dim_{H}(X) \geq n-K^{}_{M} \alpha_{\tau_A(C)} (\lvert\log (\beta_{\max}) \rvert)^{-1}>0$.
\end{thrm}
\begin{proof}
    This result follows from \Cref{prop: Winning conditions} and \autocite[Theorem 6.1]{howat2025matrixpotentialgamestructures}.
\end{proof}
    In addition to our result on patterns we can also provide a useful tool for studying intersections of sets due to the countable intersections property.
    This is useful as Hausdorff dimension is not stable under intersections.
\begin{thrm}\label{cor:intersections}
    Given $n \in \N$, consider the square metric on $\R^n$ and let $A$ be a diagonal ($n\!\times\!n$)-matrix with diagonal entries $\beta_{11},\dots,\beta_{nn} \in (0, 1/5)$.
    Moreover, let $I$ be an at most countable set and for each $i \in I$, let $C_i\subseteq B[0,1]$ be a compact set with gaps $(G_k^{(i)})_{k\in J_i}$ and $E^{(i)}$.
    Moreover, for each $i\in I$, $\tau_A(C_{i})\not\in \{-\infty,\infty\}$.
    
    Set $C = \cap_{i \in I} (C_i\cup E^{(i)})$ and $\alpha_{\tau_A(C_i)}=\prod_{j=1}^n \beta_{jj}^{\tau_A(C_i)-1}$.
    If there exists $\alpha,c,\delta \in (0,1)$ such that $\alpha^c=\sum_{i\in I} \alpha_{\tau_A(C_i)}^c$ and 
        \begin{align}\label{pattern_inequal_2}
            \alpha^c\leq \delta^2\left(1-\left(\prod_{j=1}^n\beta_{jj}\right)^{1-c}\right)<1 \quad \text{and} \quad 0<\prod\limits_{j=1}^n\left( \frac{1}{3}(1-5\beta_{jj}^{\lfloor\delta\alpha^{-1}\rfloor})\right)-8^n (1+2^{2n+1})\delta,
        \end{align}
        then, for every $y\in\R^n$, we have that $C\cap \big(A(B[0,1])+y\big)\neq \emptyset$.
        Moreover,
    \begin{align*}
        \text{dim}_H\left(C\cap \big(A(B[0,1])+y\big)\right)\geq\max\left\{ n-K_1\frac{\alpha}{\left|\log(\beta_{\max})\right|},0\right\}.
    \end{align*}
\end{thrm}
\begin{proof}
    This result follows from \Cref{prop: Winning conditions} and \autocite[Corollary 6.7]{howat2025matrixpotentialgamestructures}.    
\end{proof}
\section{An example: Self-Affine Sierpinski carpets}
We discuss the results of Section 4 using the example of self-affine Sierpinski carpets.
Note that throughout this section we use the square metric.
For a vector $\underline{r}=(r_1,\dots,r_n)\in \N^n$ with each $r_i\geq 3$ and odd, we construct the $\underline{r}$ Self-affine Sierpinski carpet by dividing $B[0,1]$ into $r_1\cdot\cdots\cdot r_n$ boxes with side length $2r_i^{-1}$ on the $i^{\text{th}}$ axis.
The middle box is then removed and this process is repeated in each surviving box ad infimum.
This can additionally be realised as the attractor of the iterated function system $\{f_{\underline{j}}:\underline{j}\in J\}$ where:
\begin{align*}
    J=\left\{\underline{j}: j_i\in [1,r_i]\cap \N \text{ and } j_i\neq \frac{1}{2}(r_i+1)\right\} \text{ and } f_{\underline{i}}(x)=\left( \frac{x_1+i_1}{r_i}-1,\dots, \frac{x_n+i_n}{r_n}-1 \right).
\end{align*}
Let $\underline{r}$ be fixed let $\beta_{jj}=1/r_j$ be the diagonal entries of the diagonal matrix $A_{\underline{r}}$ and let $t\in (0,1)$ be given.
Set $A=A_{\underline{r}}^t$
We can observe that the $\underline{r}$ Self-affine Sierpinski carpet $C_{\underline{r}}$ has gaps of the form $A^k(\open{B}[0,1])$ for each $k\in \N$ and these are the missing central rectangle.
Consider a gap $G_j$ in the $k^{\text{th}}$ stage of the construction, we have that $S_{A}(G)=\frac{t}{k}$ since $G=A^k(\open{B}[0,1])+z$ for some $z$.
Moreover, we can observe that on the $i^{\text{th}}$ axis, $G$ is at least $\frac{1}{2}(r_i-1)$ boxes of the form $A^k(B[0,1])$ from the closest gap.
Therefore, for each $i\in\{1,\dots,n\}$, we have $q=GD_A(j,C)$ satisfies $$\frac{r_i-1}{r_i^k}\geq \frac{2}{r_i^{t/q}}$$
moreover, for some $i\in\{1,\dots,n\}$ this holds with equality and hence 
$$GD_A(j,C)^{-1}=\frac{k}{t}-\frac{1}{t}\min\left\{\log_{r_i}\left(\frac{r_i-1}{2}\right)\right\}> 0.$$
Therefore, $S_A(G_j)^{-1}-GD_A(j,C_{\underline{r}})^{-1}=t^{-1}\min_{i\in\{1,\dots,n\}}\{\log_{r_i}((r_i-1)/2)\}$ and hence $$\tau_A(C_{\underline{r}})=t^{-1}\min_{i\in\{1,\dots,n\}}\left\{\log_{r_i}\left(\frac{r_i-1}{2}\right)\right\}.$$
From this and \Cref{prop: Winning conditions}, noting that the unbounded connected component of a Sierpinski carpet is $\R^n\setminus B[0,1]$, we have that $C_{\underline{r}}\cup (\R^n\setminus B[0,1])$ is $(\alpha(\underline{r},t), A_{\underline{r}}^t,0,1,1)-$winning, where $$\alpha(\underline{r},t)=\prod_{i=1}^n r_j^{t(1-\tau_{A_{\underline{r}}^t}(C_{\underline{r}}))}=\prod_{i=1}^n r_j^{t-\min_{i\in\{1,\dots,n\}}\left\{\log_{r_i}\left(\frac{r_i-1}{2}\right)\right\}}.$$
We can now apply \Cref{thm:patterns} and to determine if $C_{\underline{r}} \cup (\R^2 \setminus B[0,1])\cap B[0,1]=C_{\underline{r}}$ contains patterns for certain $\underline{r}$ values; it suffices to show that there exists a $\delta \in (0,1)$ such that the inequalities in \eqref{pattern_inequal_1} hold, with $\alpha = \alpha (\underline{r},t)$.
Using numerical methods we obtain the following.
\begin{itemize}
    \item With $\underline{r}=(2^{22},2^{23})$ we find that $C_{\underline{r}}$ contains a homothetic copy of every set with at most $37$ elements.
    \item With $\underline{r}=(2^{20},2^{21}+2^{20})$ we find that $C_{\underline{r}}$ contains a homothetic copy of every set with at most $3$ elements.
    \item With $\underline{r}=(2^{30},2^{20})$ we find that $C_{\underline{r}}$ contains a homothetic copy of every set with at most $463$ elements.
    \item With $\underline{r}=(2^{30},2^{28})$ we find that $C_{\underline{r}}$ contains a homothetic copy of every set with at most $223241$ elements.
\end{itemize}
\section*{Acknowledgements}
The author thanks Kenneth Falconer and Alexia Yavicoli for helpful discussions in writing this paper.
He also thanks Andrew Mitchell and Tony Samuel for their input and the Engineering and Physical Sciences Research Council (EPSRC) Doctoral Training Partnership and the University of Birmingham for financial support.

\printbibliography
\end{document}